\documentclass{amsart}

\usepackage[english]{babel}
\usepackage[utf8]{inputenc}
\usepackage{amssymb}
\usepackage{amsmath}
\usepackage{amsthm}
\usepackage{bbm}

\newtheorem{theorem}{Theorem}[section]
\newtheorem*{theorem_no_number}{Theorem}
\newtheorem{corollary}[theorem]{Corollary}
\newtheorem{lemma}[theorem]{Lemma}
\newtheorem{proposition}[theorem]{Proposition}
\theoremstyle{definition}
\newtheorem{definition}[theorem]{Definition}
\theoremstyle{remark}

\numberwithin{equation}{section}

\newcommand{\bbC}{\mathbb{C}}
\newcommand{\bbN}{\mathbb{N}}

\newcommand{\calL}{\mathcal{L}}

\newcommand{\re}{\operatorname{Re}}
\newcommand{\im}{\operatorname{Im}}

\begin{document}

\title[Weak Decay Rates and Uniform Stability]{On Weak Decay Rates and Uniform Stability of Bounded Linear Operators}

\author{Jochen Gl\"uck}
\address{%
Institute of Applied Analysis\\
Ulm University\\
89069 Ulm\\
Germany}

\email{jochen.glueck@uni-ulm.de}

\thanks{During his work on this article the author was supported by a scholarship of the ``Landesgraduierten-F\"orderung Baden-W\"urttemberg''.}



\date{February 11, 2015}


\begin{abstract}
	We consider a bounded linear operator $T$ on a complex Banach space $X$ and show that its spectral radius $r(T)$ satisfies $r(T) < 1$ if all sequences $(\langle x',T^nx\rangle)_{n \in \mathbb{N}_0}$ ($x \in X$, $x' \in X'$) are, up to a certain rearrangement, contained in a principal ideal of the space $c_0$ of sequences which converge to $0$. From this result we obtain generalizations of theorems of G. Weiss and J. van Neerven. We also prove a related result on $C_0$-semigroups.
\end{abstract}

\maketitle

\section{Introduction}

Let $T$ be a bounded linear operator on a complex Banach space $X$ and suppose that $T$ is weakly stable, which means that the powers $T^n$ converge to zero as $n \to \infty$ with respect to the weak operator topology. There are plenty of results in the literature stating that $T$ is in fact uniformly stable, meaning that $T^n$ converges to zero with respect to the operator norm topology, if the weak convergence satisfies certain decay properties. For instance, G.~Weiss showed in \cite[Theorem~1]{Weiss1989} that the spectral radius $r(T)$ satisfies $r(T) < 1$ provided that there is a $p \in [1,\infty)$ such that the sequence $(\langle x', T^nx\rangle)_{n \in \bbN_0}$ is contained in $l^p := l^p(\bbN_0; \bbC)$ for each $x \in X$ and each functional $x'$ in the dual space $X'$. In this note we will show, among other results, that the conclusion of Weiss' theorem remains true if $p$ is allowed to depend on $x$ and $x'$, i.e.~we prove the following result:

\begin{theorem_no_number}
	Let $T$ be a bounded linear operator on a complex Banach space $X$ and suppose that for all $x \in X$, $x' \in X'$ there exists $p \in [1,\infty)$ such that $(\langle x', T^nx \rangle)_{n \in \bbN_0} \in l^p$. Then $r(T) < 1$.
\end{theorem_no_number}

This theorem (in fact, a slight generalization of it) is shown in Corollary~\ref{cor_weak_decay_rates_in_union_of_l_p_spaces_implies_spec_rad_smaller_than_one} below. In \cite[Theorem 0.1]{Neerven1995}, van Neerven gave the following generalization of Weiss' result: Let $\varphi: [0,\infty) \to [0,\infty)$ be a non-decreasing function which is strictly positive on $(0,\infty)$ and suppose that
\begin{align}
	\sum_{n=0}^\infty \varphi(|\langle x', T^n x \rangle|) < \infty \label{form_error_series_introduction}
\end{align}
for all $x \in X$, $x' \in X'$ satisfying $||x||\le 1$, $||x'|| \le 1$. Then $r(T) < 1$. In the current note, we show that the same conclusion still holds true if one allows the function $\varphi$ to vary within a countable set; more precisely, we prove the following:

\begin{theorem_no_number}
	Let $T$ be a bounded linear operator on a complex Banach space $X$. Let $\Phi$ be an at most countable set of non-decreasing functions $\varphi: [0,\infty) \to [0,\infty)$ which are strictly positive on $(0,\infty)$ and suppose that for each $x \in X$, $x' \in X'$ satisfying $||x||\le 1$, $||x'|| \le 1$ there is a $\varphi \in \Phi$ such that (\ref{form_error_series_introduction}) holds true. Then $r(T) < 1$.
\end{theorem_no_number}

A slightly more general version of this theorem is proved in Theorem~\ref{thm_one_of_countably_many_series_smaller_then_infty_implies_spec_rad_smaller_than_one} below. The reader will easily see that the first theorem stated above follows from the second one. Besides that, we also prove a $C_0$-semigroup version of the first theorem in Section~\ref{section_stability_of_c_0_semigroups}.

Let us comment briefly on the history of the subject: the development of results like those quoted above was motivated by a well-known theorem of Datko and Pazy, who proved in \cite{Datko1970}, \cite{Pazy1972} that an integrability condition on the orbits of a $C_0$-semigroup ensures that its growth bound is strictly negative; see also \cite[Theorem~V.1.8]{Engel2000}. It was proved by G.~Weiss that a weak version of the same result holds true on Hilbert spaces \cite[Theorem~1.1]{Weiss1988} (but, as noted by van Neerven in \cite[Example~4.1]{Neerven1995}, a weak version of the Datko-Pazy Theorem cannot hold true on arbitrary Banach spaces). 
In \cite{Weiss1989}, Weiss proved a single operator analogue of the Datko-Pazy Theorem, and in \cite{Neerven1995}, van Neerven generalized the known results on single operators as well as on semigroups by considering certain Banach function spaces. For more details on the development of the subject up to then and for additional references we also refer to the introduction of \cite{Neerven1995}. Further information on the size of (weak) operator orbits can also be found in the survey article \cite{Muller2001}. Since then, a considerable number of results on this topic has appeared, many of them devoted to the analysis of more general evolution semigroups, as e.g.~in \cite{Preda2006}. We also want to mention the paper \cite{Nikolski2002}, where the resolvent of an operator $T$ is assumed to belong to some spaces of holomorphic functions, which implies $r(T) < 1$ under certain conditions. A very recent contribution to the field is \cite{Bucse2014} where uniform stability of a $C_0$-semigroup is derived from certain discrete, weak summability conditions. Some rather general results on the existence of large weak orbits of $C_0$-semigroups were recently published in \cite{Storozhuk2010} and \cite{Muller2013} \par 

While van Neerven used Banach function spaces and in particular Orlicz spaces to prove his result quoted at the beginning, we use a different approach: In Section~\ref{section_principal_ideals_and_convergence_rates}, we show that an operator $T$ satisfies $r(T) < 1$ if all sequences $(\langle x',T^nx \rangle)_{n \in \bbN_0}$ ($x \in X$, $x' \in X'$) are in some appropriate manner controlled by a sequence $f$ which converges to $0$. In Proposition~\ref{prop_countable_union_of_principal_ideals_in_c_0} we show that the situation does not change if $f$ is allowed to vary within a countable set of sequences. In Section~\ref{section_summability_of_errors} we use these results to deduce the two theorems stated above. Finally, in Section~\ref{section_stability_of_c_0_semigroups}, we adapt one of our results to the setting of $C_0$-semigroups. All our main results will be proved on complex Banach spaces, but the reader will find it easy to adapt our results to the setting of real Banach spaces by using complexifications.

\bigskip \par

\section{Preliminaries}

Let us fix some notation that will be used throughout the paper: If $X$ is a real or complex Banach space, then we denote the space of all bounded linear operators on $X$ by $\calL(X)$. Let $X$ be a complex Banach space. If $A: X \supset D(A) \to X$ is a linear operator, then $\sigma(A)$ is the \emph{spectrum} of $A$ and $s(A) := \sup\{\re \lambda: \lambda \in \sigma(A)\}$ the \emph{spectral bound} of $A$. If $\lambda \in \bbC \setminus \sigma(A)$, then $R(\lambda,A) := (\lambda-A)^{-1}$ denotes the \emph{resolvent} of $A$ at $\lambda$. If $T \in \calL(X)$, we use the symbol $r(T)$ for the \emph{spectral radius} of $T$. The operator $T$ is said to be \emph{power-bounded} if $\sup_{n \in \bbN_0} ||T^n|| < \infty$, and in this case we automatically have $r(T) \le 1$. \par 
Let $X$ be a real or complex Banach space. The \emph{dual space} of $X$ (i.e.~the space of all bounded linear functionals on $X$) will be denoted by $X'$. A vector subspace $E \subset X'$ is called \emph{almost norming} if there is a positive constant $C$ such that
\begin{align*}
	||x|| \le \sup_{x' \in E, \, ||x'|| \le 1} C \cdot |\langle x', x \rangle| \text{.}
\end{align*}
for all $x \in X$. \par 
The $C_0$-semigroup on $X$ with generator $A$ will be denoted by $(e^{tA})_{t \ge 0}$. We use the symbol $\omega_0(A)$ to denote the \emph{growth bound} (or \emph{type}) of the $C_0$-semigroup $(e^{tA})_{t \ge 0}$. If $X$ is a complex Banach space, then we denote by $s_0(A)$ the \emph{abscissa of uniform boundedness} of the resolvent of $A$, i.e.~the infimum of all real number $s$ such that the halfplane $\{z \in \bbC: \re z \ge s\}$ is contained in the resolvent set $\bbC \setminus \sigma(A)$ of $A$ and such that the resolvent $R(\cdot,A)$ is uniformly bounded on $\{z \in \bbC: \re z \ge s\}$. \par 
Let us also introduce notation for a few sequence spaces: By $c_0 := c_0(\bbN_0;\bbC)$ we denote the space of all complex-valued sequences, indexed by $\bbN_0$, which converge to $0$. By $(c_0)_+$ we denote those sequences $f = (f_n)_{n \in \bbN_0} \in c_0$ which have only real, non-negative entries $f_n$. If $g = (g_n)_{n \in \bbN_0} \in c_0$, then we call the sequence $|g| := (|g_n|)_{n \in \bbN_0}$ the \emph{modulus} of $g$. If $f \in (c_0)_+$, then the vector subspace
\begin{align*}
	(c_0)_f = \{g \in c_0: \; \exists c \ge 0: \, |g| \le c f\} \subset c_0
\end{align*}
is called the \emph{principal ideal} in $c_0$ generated by $f$. For each $f = (f_n)_{n \in \bbN_0} \in (c_0)_+$, we denote by $f^* \in (c_0)_+$ the \emph{non-increasing rearrangement} of $f$, i.e.~the unique sequence which is non-increasing and has the same entries as $f$ (including their multiplicities). The $n$-th entry of $f^*$ will be denoted by $(f^*)_n =: f_n^* =: {f_n}^*$; the last of these notations is slightly imprecise, but it is more convenient in concrete computations. If $1 \le p < \infty$, then $l^p := l^p(\bbN_0;\bbC)$ denotes the space of all complex-valued $p$-summable sequences which are indexed by $\bbN_0$. \par 

On several occasions we will need the following proposition which is a version of the Uniform Boundedness Principle.

\begin{proposition} \label{prop_unif_boundedness_on_norming_dual_sets}
	Let $X$ be a real or complex Banach space and let $E \subset X'$ be a closed, almost norming vector subspace of $X'$. Let $A \subset X$ be such that $\{\langle x', x\rangle: x \in A \}$ is bounded for every $x' \in E$. Then $A$ is bounded in $X$.
	\begin{proof}
		The proof is essentially the same as for the Uniform Boundedness Principle; we therefore leave it to the reader.
	\end{proof}
\end{proposition}

\section{Weak decay rates governed by principal ideals in $c_0$} \label{section_principal_ideals_and_convergence_rates}

The following notion will be essential in formulating conditions which imply that the powers of an operator converge to $0$ with respect to the operator norm topology.

\begin{definition} \label{def_governing_sequences}
	Let $\emptyset \not= F \subset (c_0)_+$ and let $a = (a_n)_{n \in \bbN_0}$ be a sequence of complex numbers. We say that $F$ governs $a$ if $a \in c_0$ and if there exists a sequence $f \in F$ which satisfies $(|a_n|^*)_{n \in \bbN_0} \in (c_0)_f$. We say that a sequence $f \in (c_0)_+$ governs $a$ if the set $\{f\}$ governs $a$.
\end{definition}

The following theorem states that the powers of an operator $T$ converge to zero with respect to the operator norm topology if the non-increasing rearrangements of the sequences $(|\langle x', T^nx \rangle|)_{n \in \bbN_0}$ decay with the same rate for all $x \in X$, $x'\in X'$.

\begin{theorem} \label{thm_governing_sequence_for_weak_decay_rates_implies_spec_rad_smaller_than_one}
	Let $X$ be a complex Banach space, let $E \subset X'$ be a closed, almost norming vector subspace and let $T \in \calL(X)$. Let $f \in (c_0)_+$ and suppose that $f$ governs $(\langle x', T^nx \rangle)_{n \in \bbN_0}$ for all $x \in X$, $x' \in E$. Then $r(T) < 1$.
\end{theorem}

For the proof we need the following infinite series version of the rearrangement inequality.

\begin{lemma} \label{lem_rearrangement_inequality_infinite}
	Let $g = (g_n)_{n \in \bbN_0} \in l^1$ be a non-increasing sequence of non-negative real numbers and let $f = (f_n)_{n \in \bbN_0} \in (c_0)_+$. Then we have
	\begin{align*}
		\sum_{n=0}^\infty f^*_n g_n \ge \sum_{n=0}^\infty f_n g_n \text{.}
	\end{align*}
	\begin{proof}
		This follows by an approximation argument from the rearrangement inequality for finite sums \cite[Theorem 368]{Hardy1952}.
	\end{proof}
\end{lemma}

\begin{proof}[Proof of Theorem~\ref{thm_governing_sequence_for_weak_decay_rates_implies_spec_rad_smaller_than_one}]
	(a) By assumption, the sequence $(\langle x', T^n x \rangle)_{n \in \bbN_0}$ is contained in $c_0$, and in particular bounded, for each $x\in X$, $x' \in E$. Hence, it follows from Proposition~\ref{prop_unif_boundedness_on_norming_dual_sets} and from the Uniform Boundedness Principle that $T$ is power-bounded. Thus, $r(T) \le 1$ and we have to show that there is no spectral value of $T$ located on the complex unit circle. We assume for a contradiction that $\lambda \in \bbC$, $|\lambda| = 1$, is a spectral value of $T$. \par 
	(b) For every $r > 1$, we use the abbreviation $e(r) := (r-1) \sum_{n=0}^\infty \frac{f_n}{r^{n+1}}$. We may assume that $f \not= 0$, and thus we have $e(r) > 0$ for all $r > 1$. If $x \in X$ and $x' \in E$, then we have $|\langle x', T^n x \rangle|^* \le c f_n$ for all $n \in \bbN_0$ and some number $c > 0$ (where $c$ might depend on $x$ and $x'$). Using the Neumann series representation of $R(r\lambda,T)$ and the rearrangement inequality from Lemma~\ref{lem_rearrangement_inequality_infinite}, we obtain the estimate
	\begin{align*}
		|\langle x', R(r\lambda,T) x \rangle| \le \sum_{n=0}^\infty \frac{|\langle x', T^n x \rangle|}{r^{n+1}} \le \sum_{n=0}^\infty \frac{|\langle x', T^n x \rangle|^*}{r^{n+1}} \le c \sum_{n=0}^\infty \frac{f_n}{r^{n+1}} = c \frac{e(r)}{r-1}
	\end{align*}
	whenever $r > 1$. By Proposition~\ref{prop_unif_boundedness_on_norming_dual_sets} and by the Uniform Bounded Principle we conclude that the operator family $\{\frac{r-1}{e(r)}R(r\lambda,T): r > 1\}$ is bounded. On the other hand, we have
	\begin{align*}
		||\frac{r-1}{e(r)} R(r\lambda,T)|| \ge \frac{r-1}{e(r)}\frac{1}{\operatorname{dist}(r\lambda, \sigma(T))} = \frac{1}{e(r)}
	\end{align*}
	for each $r > 1$. \par 
	(c) To obtain a contradiction, let us now show that $e(r) \to 0$ as $r \downarrow 1$: If $\varepsilon > 0$, then we can find an integer $n_0 \in \bbN$ such that $f_n \le \varepsilon$ for all $n \ge n_0$. Set $\delta := \frac{\varepsilon}{n_0 ||f||_\infty}$. Then we have for all $r \in (1,1+\delta)$ that
	\begin{align*}
		e(r) \le (r-1) & \sum_{n=0}^{n_0-1} f_n + (r-1) \sum_{n=n_0}^\infty \frac{\varepsilon}{r^{n+1}} \le \\
		& \le \delta n_0 ||f||_\infty + (r-1) \sum_{n=0}^\infty \frac{\varepsilon}{r^{n+1}} = 2 \varepsilon \text{.}
	\end{align*}
	Hence we have indeed $e(r) \to 0$ as $r \downarrow 1$ and thus we arrived at the desired contradiction.
\end{proof}

The main idea of the above proof is to use the growth behaviour of the resolvent close to the spectral circle; the same approach was used by Weiss to prove the main result in \cite{Weiss1989}.

In fact, one may allow the decay rates of the sequences $(\langle x', T^nx \rangle)_{n \in \bbN_0}$ in Theorem~\ref{thm_governing_sequence_for_weak_decay_rates_implies_spec_rad_smaller_than_one} to vary within a \emph{countable} set of possible rates; this follows from the next proposition and its corollary.

\begin{proposition} \label{prop_countable_union_of_principal_ideals_in_c_0}
	Let $\emptyset \not= F \subset (c_0)_+$ be an at most countable set. Then there is a sequence $g \in (c_0)_+$ such that $(c_0)_f \subset (c_0)_g$ for all $f \in F$.
	\begin{proof}
		We may assume that $f \not= 0$ for all $f \in F$. Let $F = \{f^{(k)}: k \in \bbN\}$ (where some of the functions $f^{(k)}$ coincide if $F$ is finite). Now, define $g = \sum_{k=0}^\infty \frac{1}{2^k||f^{(k)}||_\infty} f^{(k)}$, where the series converges with respect to the supremum norm in $c_0$. Then clearly $g \in (c_0)_+$. Moreover, we obtain for each $k \in \bbN$ that $f^{(k)} \le 2^k||f^{(k)}||\, g$, which implies $(c_0)_{f^{(k)}} \subset (c_0)_g$.
	\end{proof}
\end{proposition}

\begin{corollary} \label{cor_coutable_governing_set_for_weak_decay_rates_implies_spec_rad_smaller_then_one}
	Let $X$ be a complex Banach space, let $E \subset X'$ be a closed, almost norming vector subspace and let $T \in \calL(X)$. Let $\emptyset \not= F \subset (c_0)_+$ be an at most countable set and suppose that $F$ governs $(\langle x',T^nx \rangle)_{n \in \bbN_0}$ for all $x \in X$, $x' \in E$. Then $r(T) < 1$.
	\begin{proof}
		Due to Proposition~\ref{prop_countable_union_of_principal_ideals_in_c_0}, there is an element $f \in (c_0)_+$ which governs the sequence $(\langle x',T^nx \rangle)_{n \in \bbN_0}$ for all $x \in X$, $x' \in E$. Hence, the assertion follows from Theorem~\ref{thm_governing_sequence_for_weak_decay_rates_implies_spec_rad_smaller_than_one}.
	\end{proof}
\end{corollary}

\section{Weak decay rates governed by summability conditions} \label{section_summability_of_errors}

To ensure that the conditions of Theorem~\ref{thm_governing_sequence_for_weak_decay_rates_implies_spec_rad_smaller_than_one} or Corollary~\ref{cor_coutable_governing_set_for_weak_decay_rates_implies_spec_rad_smaller_then_one} are satisfied, one can e.g.~impose some kind of summability condition on the sequences $(\langle x', T^nx \rangle)_{n \in \bbN_0}$. This is the content of the next lemma and the subsequent theorem.

\begin{lemma} \label{lem_series_smaller_then_infty_yields_governing_sequence}
	Let $A$ be a set of complex-valued sequences $a = (a_n)_{n \in \bbN_0}$ and let $\varphi: [0,\infty) \to [0,\infty)$ be a non-decreasing function which is strictly positive on $(0,\infty)$. If
	\begin{align*}
		\sum_{n=0}^\infty \varphi(|a_n|) < \infty
	\end{align*}
	for all $a \in A$, then there is an $f \in (c_0)_+$ which governs all elements of $A$.
	\begin{proof}
		(a) Since $\varphi$ is strictly positive on $(0,\infty)$, we conclude that $A \subset c_0$. \par 
		(b) Let $c_{00}$ be the set of all sequences with only finitely many non-zero entries. If $A \subset c_{00}$, then $A \subset (c_0)_f$ for any $f \in (c_0)_+$ which has no zero-entries, so we may assume for the rest of the proof that there is at least one element $b = (b_n)_{n \in \bbN_0} \in A \setminus c_{00}$. \par 
		Let us show that this implies $\varphi(x) \downarrow 0$ as $x \downarrow 0$: If $\varepsilon > 0$, then we can find an index $n_0 \in \bbN_0$ such that $\varphi(|b_n|) < \varepsilon$ for all $n \ge n_0$. Since $b \not\in c_{00}$, we have $\delta := |b_{n_1}| > 0$ for some $n_1 \ge n_0$. As $\varphi$ is non-decreasing, we conclude that $\varphi(x) \le \varphi(\delta) < \varepsilon$ for all $x \in [0,\delta)$. Hence, we have indeed $\varphi(x) \downarrow 0$ as $x \downarrow 0$ \par 
		(c) Next, we construct the sequence $f$: Let $(m_k)_{k \in \bbN} \subset \bbN_0$ be a strictly increasing sequence of integers such that
		\begin{align*}
			m_k \varphi(\frac{1}{k}) \ge 1 \quad \text{for each } k \in \bbN \text{;}
		\end{align*}
		such a sequence $(m_k)_{k \in \bbN}$ clearly exists since $\varphi(\frac{1}{k}) > 0$ for all $k \in \bbN$. We define $f \in (c_0)_+$ by $f(j) = 1$ for $0 \le j \le m_1-1$ and $f(j) = \frac{1}{k-1}$ for $m_{k-1} \le j \le m_k - 1$ whenever $k \ge 2$. \par 
		(d) Let $a \in A$, and first assume that $\sum_{n=0}^\infty \varphi(|a_n|) \le 1$. Then for each $k \in \bbN$ at most $m_k$ entries of the sequence $|a|$ can be equal to or larger then $\frac{1}{k}$. Keeping in mind that the indices of the sequence $a$ start at $0$ we conclude that $|a_n|^* \le f(n)$ for each $n \ge m_1$ and hence we have $|a|^* \le cf$ for some $c > 0$. \par 
		Now only assume that $\sum_{n=0}^\infty \varphi(|a_n|) < \infty$. Let $\mu \in (0,1]$ and observe that, due to (b), $\varphi(\mu |a_n|) \downarrow 0$ as $\mu\downarrow 0$ for each $n \in \bbN_0$. By the Dominated Convergence Theorem, this implies that
		\begin{align*}
			\sum_{n=0}^\infty \varphi(\mu |a_n|) \downarrow 0 \quad \text{as} \quad \mu \downarrow 0 \text{.}
		\end{align*}
		Hence, we can find a number $\mu \in (0,1]$ such that $\sum_{n=0}^\infty \varphi(\mu |a_n|) \le 1$. As shown before, this implies that $\mu |a|^* = (\mu |a|)^* \le c f$ for some $c > 0$ and hence $|a|^* \le \frac{c}{\mu}f$. Thus $f$ governs $a$.
	\end{proof}
\end{lemma}

We now immediately obtain our main result, which is a generalization of \cite[Theorem~3.3]{Neerven1995}.

\begin{theorem} \label{thm_one_of_countably_many_series_smaller_then_infty_implies_spec_rad_smaller_than_one}
	Let $X$ be a complex Banach space, let $E \subset X'$ be a closed, almost norming subspace and let $T \in \calL(X)$. Let $\Phi$ be an at most countable set of non-decreasing functions $\varphi: [0,\infty) \to [0,\infty)$ which are strictly positive on $(0,\infty)$. If for each $x \in X$, $x' \in E$ satisfying $||x|| \le 1$, $||x'|| \le 1$ there is a function $\varphi \in \Phi$ which satisfies
	\begin{align}
		\sum_{n=0}^\infty \varphi(|\langle x', T^n x \rangle|) < \infty \text{,} \label{form_summability_condition_in_general_theorem}
	\end{align}
	then $r(T) < 1$.
	\begin{proof}
		It follows from Lemma~\ref{lem_series_smaller_then_infty_yields_governing_sequence} that there is an at most countable set $\emptyset \not= F \subset (c_0)_+$ which governs $(\langle x', T^nx \rangle)_{n \in \bbN_0}$ for all $x \in X$, $x' \in E$ satisfying $||x|| \le 1$, $||x'|| \le 1$. By a scaling argument, $F$ governs $(\langle x', T^nx \rangle)_{n \in \bbN_0}$ even for all $x \in X$, $x' \in E'$. Corollary~\ref{cor_coutable_governing_set_for_weak_decay_rates_implies_spec_rad_smaller_then_one} thus implies that $r(T) < 1$.
	\end{proof}
\end{theorem}

\begin{corollary} \label{cor_weak_decay_rates_in_union_of_l_p_spaces_implies_spec_rad_smaller_than_one}
	Let $X$ be a complex Banach space, let $E \subset X'$ be a closed, almost norming subspace and let $T \in \calL(X)$. Suppose that for all $x \in X$, $x' \in E$ there is a $p \in [1,\infty)$ such that $(\langle x', T^nx \rangle)_{n \in \bbN_0} \in l^p$. Then $r(T) < 1$.
	\begin{proof}
		Let $1 \le p_k \to \infty$, define $\varphi_k: [0,\infty) \to [0,\infty)$, $\varphi_k(x) = x^{p_k}$ for all $k \in \bbN$ and let $\Phi = \{\varphi_k: k \in \bbN\}$. For each $x \in X$, $x' \in E$ we can find a number $p \in [1,\infty)$ such that $(\langle x', T^nx \rangle)_{n \in \bbN_0} \in l^p$; if $k \in \bbN$ is sufficiently large we have $l^p \subset l^{p_k}$ which implies that condition (\ref{form_summability_condition_in_general_theorem}) us fulfilled for the function $\varphi_k$. Hence, the assertion follows from Theorem~\ref{thm_one_of_countably_many_series_smaller_then_infty_implies_spec_rad_smaller_than_one}.
	\end{proof}
\end{corollary}

\section{Stability of $C_0$-semigroups} \label{section_stability_of_c_0_semigroups}

Our techniques from Section~\ref{section_principal_ideals_and_convergence_rates} relied on controlling the weak decay rate of operator powers be certain ideals in $c_0$. It seems unclear how to adapt these techniques to the setting of $C_0$-semigroups. However, we can use a more concrete approach to show the following result, which generalizes results from \cite{Weiss1988}. Recall that for a $C_0$-semigroup $(e^{tA})_{t \ge 0}$, we denote by $s_0(A)$ the abscissa of uniform boundedness of the resolvent of $A$.

\begin{theorem} \label{thm_finite_weak_p_integral_for_semigroup_yields_spectral_estimate}
	Let $(e^{tA})_{t \ge 0}$ be a $C_0$-semigroup on a complex Banach space $X$. Let $E \subset X'$ be a closed, almost norming subspace and suppose that for each $x \in X$, $x' \in E$ there is a $p \in [1,\infty)$ such that
	\begin{align}
		\int_0^\infty |\langle x', e^{tA}x \rangle|^p \, dt < \infty \text{.} \label{form_finite_integral_p_norm}
	\end{align}
	Then $s_0(A) < 0$.
\end{theorem}

If $X$ is a Hilbert space, then the conclusion of the above theorem implies that the growth bound of the semigroup satisfies $\omega_0(A) < 0$; this follows from the Gearhart-Pr\"uss Theorem \cite[Theorem~V.1.11]{Engel2000}. However, note that the conditions of Theorem~\ref{thm_finite_weak_p_integral_for_semigroup_yields_spectral_estimate} cannot be sufficient to ensure that $\omega_0(A) < 0$ on an arbitrary Banach space $X$, even if estimate~(\ref{form_finite_integral_p_norm}) holds true for fixed $p=1$ and for all $x\in X$, $x' \in X'$; for a counterexample we refer to \cite[Example~4.1]{Neerven1995}. \par 
We should point out that, at least in the case $E = X'$, Theorem~\ref{thm_finite_weak_p_integral_for_semigroup_yields_spectral_estimate} can be obtained as a corollary from a recent result of Storozhuk on the size of weak semigroup orbits \cite[Theorem~1]{Storozhuk2010}. However, since the proof of Storozhuk's result is technically rather involved, we prefer to give a more direct argument here. Our proof is a modification of the arguments given in \cite{Weiss1988}.

\begin{proof}[Proof of Theorem~\ref{thm_finite_weak_p_integral_for_semigroup_yields_spectral_estimate}]
	We may assume throughout the proof that the growth bound of the semigroup satisfies $\omega_0(A) \le \frac{1}{2}$ (otherwise we can replace $A$ by $\varepsilon A$ for some small $\varepsilon > 0$). \par 
	(a) By H\"older's inequality and assumption (\ref{form_finite_integral_p_norm}) we have
	\begin{align*}
		\int_0^\infty |\langle x', e^{-\lambda t}e^{tA}x \rangle| \, dt < \infty 
	\end{align*}
	for all $x \in X$, $x' \in E$ and for each $\lambda \in \bbC$ satisfying $\re \lambda > 0$. We can now use a similar argument as in the proof of \cite[Proposition~1.1]{Greiner1981} to show that the limit
	\begin{align}
		R(\lambda) := \lim_{\tau \to \infty} \int_0^\tau e^{-\lambda t}e^{tA} \, dt \label{form_integral_converges_to_resolvent}
	\end{align}
	exists with respect to the operator norm topology on $\calL(X)$ whenever $\re \lambda > 0$  (here, the integral $\int_0^\tau e^{-\lambda t}e^{tA} \, dt$ is understood in the strong sense, of course). For the sake of completeness, we include the argument here: Let $\re \lambda > 0$ and choose any $\mu \in \bbC$ such that $\re \mu \in (0,\re \lambda)$. Consider the bilinear map
	\begin{align*}
		b_\mu: X \times E \to L^1([0,\infty);\bbC) \text{,} \quad (x,x') \mapsto (t \mapsto \langle x', e^{-\mu t}e^{tA}x \rangle) \text{.}
	\end{align*}
	Then for each $x_0\in X$, $x'_0 \in E$, both partial mappings $b_\mu(\cdot, x'_0)$ and $b_\mu(x_0, \cdot)$ have closed graph and are hence continuous. Thus, the bilinear map $b_\mu$ is continuous (see \cite[Section III.5.1, Corollary 1]{Schaefer1999}), i.e.~there is a constant $M_\mu \ge 0$ such that
	\begin{align*}
		\int_{0}^\infty |\langle x', e^{-t\mu}e^{tA}x \rangle | \, dt \le M_\mu ||x|| \, ||x'||
	\end{align*}
	for all $x \in X$, $x' \in E$. For $0 \le \tau_1 \le \tau_2$ we thus obtain
	\begin{align*}
		\int_{\tau_1}^{\tau_2} |\langle x', e^{-t\lambda}e^{tA} x \rangle | \, dt & = \int_{\tau_1}^{\tau_2} e^{-t(\re \lambda - \re \mu)}|\langle x', e^{-t\mu}e^{tA} x \rangle | \, dt \le \\
		& \le e^{-\tau_1(\re \lambda - \re \mu)} M_\mu ||x|| \, ||x'||
	\end{align*}
	for all $x \in X$, $x' \in E$. Since $E$ is almost norming, we can find a constant $C \ge 0$ such that
	\begin{align*}
		||\int_{\tau_1}^{\tau_2} e^{-t\lambda}e^{tA} \, dt || \le e^{-\tau_1(\re \lambda - \re \mu)} M_\mu C
	\end{align*}
	whenever $0 \le \tau_1 \le \tau_2$. Hence, the net $(\int_0^\tau e^{-\lambda t}e^{tA} \, dt)_{\tau \ge 0}$ is a Cauchy net in $\calL(X)$ and thus converges. \par
	From the convergence with respect to the norm topology in (\ref{form_integral_converges_to_resolvent}) we can conclude that each $\lambda \in \bbC$ which satisfies $\re \lambda > 0$  is contained in the resolvent set of $A$ and that we have $R(\lambda,A) = R(\lambda)$ for each such $\lambda$ (see \cite[Theorem~II.1.10~(i)]{Engel2000}). \par 
	(b) Let $\lambda \in \bbC$, $\re \lambda \in (0,1)$ and let $x\in X$, $x' \in E$. Choose $p \in [1,\infty)$ such that the estimate~(\ref{form_finite_integral_p_norm}) is satisfied and let $q \in (1,\infty]$ be the conjugate index of $p$, i.e.~let $\frac{1}{p} + \frac{1}{q} = 1$. Let $C_p(x',x) := (\int_0^\infty |\langle x', e^{tA}x \rangle|^p \, dt)^{1/p}$, and let $M_q(\re \lambda)$ be the $L^q$-norm of the mapping $[0,\infty) \to [0,\infty)$, $t \mapsto e^{-t\re \lambda}$. H\"older's inequality together with (a) then yields
	\begin{align*}
		|\langle x', R(\lambda,A) x \rangle| \le M_q(\re \lambda) C_p(x',x) \text{.}
	\end{align*}
	A short computation shows that
	\begin{align*}
		M_q(\re \lambda) =
		\begin{cases}
			(\re \lambda \, q)^{-\frac{1}{q}} \le (\re \lambda)^{-\frac{1}{q}} \quad & \text{if } q \in (1,\infty) \\
			1 \quad & \text{if } q = \infty \text{.}
		\end{cases}
	\end{align*}
	In any case, we have that $\re\lambda \cdot |\log(\re \lambda)|\cdot  M_q(\re \lambda) \to 0$ as $\re \lambda \downarrow 0$. Hence, $|\langle x', \re \lambda \cdot \log(\re \lambda) R(\lambda,A)x \rangle |$ is bounded for $\re\lambda \in (0,1]$. Proposition~\ref{prop_unif_boundedness_on_norming_dual_sets} and the Uniform Boundedness Principle hence yield a constant $M > 0$ such that
	\begin{align}
		||R(\lambda, A)|| \le \frac{M}{\re \lambda \cdot |\log(\re \lambda)|} \label{form_resolvent_estimate}
	\end{align}
	for all $\lambda \in \bbC$ with $\re \lambda \in (0,1)$. \par 
	(c) For some sufficiently small $r \in (0,1)$ we have $4M < |\log\frac{r}{4}|$ and thus $r < \frac{r}{4M}|\log \frac{r}{4}|$. Now, let $\mu \in \bbC$ such that $\re \mu \in [-\frac{r}{4}, \frac{r}{4}]$ and define $\lambda := \frac{r}{4} + i\im \mu$. By (\ref{form_resolvent_estimate}) we then have 
	\begin{align*}
		|\lambda - \mu| \le \frac{r}{2} < \frac{r}{4M} \big|\log \frac{r}{4}\big| - \frac{r}{2} \le \frac{1}{||R(\lambda,A)||} - \frac{r}{2} \text{.}
	\end{align*}
	Since $|\lambda - \mu| < \frac{1}{||R(\lambda,A)||}$, the number $\mu$ is contained in the resolvent set of $A$ and we have $R(\mu,A) = R(\lambda,A)\sum_{n=0}^\infty (\lambda - \mu)^n R(\lambda,A)^n$. We thus obtain the estimate
	\begin{align*}
		||R(\mu,A)|| \le \frac{||R(\lambda,A)||}{1 - |\lambda-\mu|\cdot ||R(\lambda,A)||} = \frac{1}{\frac{1}{||R(\lambda,A||} - |\lambda-\mu|} \le \frac{2}{r} \text{.}
	\end{align*}
	Therefore, the strip $\{z \in \bbC: \re z \in [-\frac{r}{4}, \frac{r}{4}]\}$ is contained in the resolvent set of $A$ and the resolvent $R(\cdot,A)$ is uniformly bounded on this strip. Moreover, it follows from estimate~(\ref{form_resolvent_estimate}) that $R(\lambda,A)$ is uniformly bounded for $\re \lambda \in [\frac{r}{4},\frac{2}{3}]$, and since $\omega_0(A) \le \frac{1}{2}$, $R(\lambda,A)$ is also uniformly bounded for $\re \lambda \ge \frac{2}{3}$. Hence, we conclude that $s_0(A) \le -\frac{r}{4} < 0$.
\end{proof}

\paragraph*{Acknowledgements} I wish to thank Yuri Tomilov for pointing out references \cite{Storozhuk2010} and \cite{Muller2013} to me.

\end{document}